\UseRawInputEncoding
\documentclass[11pt]{amsart}
\usepackage{amsfonts}
\usepackage{}
\usepackage{mathrsfs}
\usepackage{bbm}
\usepackage{amssymb}
\usepackage{indentfirst}
\usepackage{latexsym,amsfonts,amssymb,amsmath}
\pdfoutput=1%For ARXIV
\newtheorem{theorem}{Theorem}[section]
\newtheorem{lemma}[theorem]{Lemma}
\newtheorem{corollary}[theorem]{Corollary}
\newtheorem{question}[theorem]{Question}
\theoremstyle{definition}
\newtheorem{definition}[theorem]{Definition}
\newtheorem{proposition}[theorem]{Proposition}
\theoremstyle{remark}

\begin{document}

\title
{On the continuity of the inverse in (strongly) paratopological gyrogroups}
\author{Ying-Ying Jin}\thanks{}
\address{(Y.-Y. Jin) Department of General Required Courses, Guangzhou Panyu Polytechnic, Guangzhou 511483, P.R. China} \email{yingyjin@163.com, jinyy@gzpyp.edu.cn}

\author{Li-Hong Xie*}\thanks{* The corresponding author.}
\address{(L.-H. Xie) School of Mathematics and Computational Science, Wuyi University, Jiangmen 529020, P.R. China} \email{yunli198282@126.com}

\thanks{
This work is supported by the Natural Science Foundation of Guangdong
Province under Grant (Nos. 2021A1515010381; 2020A1515110458). The Innovation Project of Department of Education of Guangdong Province (No. 2022KTSCX145), and the Natural Science Project of Jiangmen City (No:2021030102570004880).}

\subjclass[2010]{primary 54H99; secondary 54D30, 54D45, 54D20, 	54B15, 54C10}

\keywords{Paratopological gyrogroup; Topological gyrogroup; Compact; Locally compact; Feebly compact; Pseudocompact}

\begin{abstract} In this paper, we consider the continuity of the inverse in (strongly) paratopological gyrogroups. The conclusions are established as follows:
(1) A compact Hausdorff paratopological gyrogroup $G$ is a topological gyrogroup. (2) A Hausdorff locally compact strongly paratopological gyrogroup is a topological gyrogroup.
(3) If $G$ is locally compact strongly paratopological gyrocommutative gyrogroup (without any separation restrictions), then
$G$ is a strongly topological gyrogroup.
(4) Every regular feebly compact strongly paratopological gyrogroup is a topological gyrogroup.
(5) If a Hausdorff strongly paratopological gyrogroup $G$ is countablly compact and topologically periodic,
then $G$ is a strongly topological gyrogroup.
\end{abstract}

\maketitle

\section{Introduction}
Finding a natural compactness-type condition on a topological semigroup (or paratopological group) that appear to suggest it is a topological group has many precedents in the literature.
According to Ellis' theorem in \cite{El}, every locally compact Hausdorff semitopological group is a topological group.
Romaguera and Sanchis \cite{RS} generalized the famous Numakura¡¯s theorem \cite{Nu} and showed that
every compact Hausdorff topological semigroup with two-sided cancellation is a topological group.
A conclusion drawn from this result in \cite{RS} is that every compact $T_0$ paratopological group is a topological group.
It turns out that in the latter situation, the $T_0$ constraint can be dropped.
Ravsky \cite{Ra1} proved that a compact paratopological group is a topological group.
Ellis \cite{El1}, Grant \cite{Gr}, Brand \cite{Br}, Bouziad \cite{Bo}, Bokalo and Guran \cite{Bok},
Romaguera and Sanchis \cite{RS}, Kenderov et al.\cite{Ke}, and others have all generalized the latter fact.
Reznichenko investigated automatic continuity in paratopological groups in \cite{Rez}, proving that every completely regular pseu-docompact paratopological group $G$ is a topological group, i.e., the inversion in G is continuous.
This result was extended to regular pseudocompact paratopological groups by Arhangelski\v{\i} and Reznichenko in \cite{AR}, \cite{Arha}.

Ungar discovered the concept of gyrogroups and discussed some properties of them in 2008 when he studied the $c$-ball of relativistically admissible velocities with the Einstein velocity addition in \cite{Ung}.
The Einstein velocity addition $\oplus_E$ in the $c$-ball is given
by the equation
$$\mathbf{u}\oplus_E\mathbf{v}=\frac{1}{1+\frac{<\mathbf{u},\mathbf{v}>}{c^2}}\{\mathbf{u}+\frac{1}{\gamma_{\mathbf{u}}}\mathbf{v}+
\frac{1}{c^2}\frac{\gamma_{\mathbf{u}}}{1+\gamma_{\mathbf{u}}}<\mathbf{u}, \mathbf{v}>\mathbf{u}\},$$
where $\mathbf{u}, \mathbf{v}\in\mathbb{R}_{c}^3=\{\mathbf{v}\in\mathbb{R}^3:\|\mathbf{v}\|<c\}$ and $\gamma_{\mathbf{u}}$ is the Lorentz factor given by
$$\gamma_{\mathbf{u}}=\frac{1}{\sqrt{1-\frac{\|\mathbf{u}\|^2}{c^2}}}.$$
The operator $\oplus_E$ doesn't satisfy associative or commutative law, so $(\mathbb{R}_{c}^3, \oplus_E)$ is not a group. The associative law is redefined
by more general definitions which are the left gyroassociative law and
the left loop property.
A gyrogroup, broadly defined, is a group-like structure where the associative law does not hold (see Definition \ref{Def:gyr}).
Atiponrat \cite{Atip} recently developed the idea of topological gyrogroups as a generalization of topological groups.
A paratopological gyrogroup is a gyrogroup with a topology such that its binary operation is
jointly continuous. If $G$ is a paratopological gyrogroup and the inverse operation of $G$ is continuous, then $G$ is
a topological gyrogroup.
Specially, Atiponrat \cite{Atip} discovered that for a topological gyrogroup, $T_0$ and $T_3$ are equivalent. It is worth noting that Cai, Lin and He in \cite{Cai} proved that every Hausdorff first countable topological gyrogroup is metrizable.
Atiponrat and Maungchang \cite{Atip1} studied some separation axioms of paratopological gyrogroups. In \cite{JX1}, Jin and Xie proved that every regular (Hausdorff) locally gyroscopic invariant paratopological gyrogroup $G$ is completely regular (function Hausdorff), and extended the Pontrjagin conditions of (para)topological groups to (para)topological gyrogroups.

As a generalization of paratopological groups, it is natural to consider the conditions for a partopological gyrogroup
to turn out to be a topological gyrogroup.
In this paper, we try to study whether a paratopological
gyrogroup satisfying a natural compactness-type condition and a separation axiom turns out to be a topological gyrogroup.
The paper is organized as follows: In Section 2, we mainly introduce the related concepts and conclusions which are required in this article.
In Section 3, we study the continuity of the inverse in (strongly) paratopological gyrogroups. The following results are established.
(1) A compact Hausdorff paratopological gyrogroup $G$ is a topological gyrogroup (see Theorem \ref{the3.2}).
(2) If $G$ is a Hausdorff
locally compact strongly paratopological gyrogroup, then $G$ is a strongly topological gyrogroup(see Theorem \ref{the3.6}).
(3) Let $G$ be a strongly paratopological gyrogroup,
and $H$ be an invariant subgyrogroup of $G$.
If $H$ and $G/H$ are strongly topological gyrogroups, then so is $G$(see Theorem \ref{the3.12}).
(4) If $G$ is locally compact strongly paratopological gyrocommutative gyrogroup (without any separation restrictions), then
$G$ is a strongly topological gyrogroup(see Theorem \ref{the3.17}).
In Section 4, we consider feebly compact paratopological gyrogroups.
The following results are established.
(1) If a strongly paratopological gyrogroup $G$ is a dense $G_\delta$-set in a regular feebly compact space $X$,
then $G$ is a strongly topological gyrogroup(see Theorem \ref{the3.8}).
(2) If a strongly paratopological gyrogroup $G$ is Hausdorff countable compact and topologically periodic,
then $G$ is a strongly topological gyrogroup(see Theorem \ref{the4.2}).

No separation restrictions on the topological spaces considered in this paper are imposed unless we mention them explicitly.

\section{Definitions and preliminaries}
\begin{definition}\cite{Ung}\label{Def:gyr}
 Let $(G, \oplus)$ be a nonempty groupoid. We say that $(G, \oplus)$ or just $G$
(when it is clear from the context) is a gyrogroup if the followings hold:
\begin{enumerate}
\item[($G1$)] There is an identity element $0 \in G$ such that
$$0\oplus x=x=x\oplus 0\text{~~~~~for all~~}x\in G.$$
\item[($G2$)] For each $x \in G $, there exists an {\it inverse element}  $\ominus x \in G$ such that
$$\ominus x\oplus x=0=x\oplus(\ominus x).$$
\item[($G3$)] For any $x, y \in G $, there exists an {\it gyroautomorphism} $\text{gyr}[x, y] \in Aut(G,  \oplus)$ such that
$$x\oplus (y\oplus z)=(x\oplus y)\oplus \text{gyr}[x, y](z)$$ for all $z \in G$;
\item[($G4$)] For any $x, y \in G$, $\text{gyr}[x \oplus y, y]=\text{gyr}[x, y]$.
\end{enumerate}
\end{definition}
For a gyrogroup $G$ and $x_1, x_2, \cdot\cdot\cdot, x_k\in G$, the formula $(((x_1\oplus x_2)\oplus x_3)\oplus \cdot\cdot\cdot\oplus x_{k-1})\oplus x_k$ will be denoted by $x_1\oplus x_2\oplus \cdot\cdot\cdot\oplus x_k$.

\begin{definition}\cite{Ung}\label{com}
A gyrogroup $(G,\oplus)$ is gyrocommutative if its binary operation obeys the gyrocommutative law
$$a\oplus b=\text{gyr}[a,b](b\oplus a)$$
for all $a, b\in G$.
\end{definition}

\begin{definition}\cite{Ung}\label{defbox}
Let $(G,\oplus)$ be a gyrogroup with gyrogroup operation (or,
addition) $\oplus$. The gyrogroup cooperation (or, coaddition) $\boxplus$ is a second
binary operation in $G$ given by the equation
$$(\divideontimes)~~~~a\boxplus b=a\oplus \text{gyr}[a,\ominus b]b$$ for all $a, b\in G$.
The groupoid $(G, \boxplus)$ is called a cogyrogroup, and is said to
be the cogyrogroup associated with the gyrogroup $(G, \oplus)$.

Replacing $b$ by $\ominus b$ in $(\divideontimes)$, along with $(\divideontimes)$ we have the identity
$$a\boxminus b=a\ominus \text{gyr}[a,b]b$$ for all $a, b\in G$, where we use the obvious notation, $a\boxminus b = a\boxplus(\ominus b)$.
\end{definition}

\begin{definition}\cite{Suk3}
Let $(G, \oplus)$ be a gyrogroup. A nonempty subset $H$ of $G$ is called a subgyrogroup, denoted
by $H\leq G$, if the following statements hold:
\begin{enumerate}
\item[(1)] The restriction $\oplus|_{H\times H}$ is a binary operation on $H$, i.e. $(H, \oplus|_{H\times H})$ is a groupoid;
\item[(2)] For any $x, y\in H$, the restriction of $\text{gyr}[x, y]$ to $H$, $\text{gyr}[x, y]|_H: H \rightarrow\text{gyr}[x, y](H)$, is a bijective
homomorphism; and
\item[(3)] $(H, \oplus|_{H\times H})$ is a gyrogroup.
\end{enumerate}
\end{definition}

Furthermore, a subgyrogroup $H$ of $G$ is said to be an {\it $L$-subgyrogroup} \cite{Suk3}, denoted by $H\leq_L G$,
if $\text{gyr}[a, h](H) =H$ for all $a\in G$ and $h\in H$.

A {\it semigroup} is a non-void set $S$ together
with a mapping $(x, y)\rightarrow xy$ of $S\times S$ to $S$ such that $x(yz)=(xy)z$ for all $x, y, z$ in $S$.

\begin{proposition}\cite{ST}
A nonempty subset $H$ of $G$ is a
subgyrogroup if and only if $a\in H$ implies $\ominus a\in H$ and $a, b\in H$ implies $a\oplus b\in H$.
\end{proposition}

\begin{proposition}\cite{ST}
 A nonempty subset $X$ of a gyrogroup $G$ is a subgroup if and only
if it is a subgyrogroup of $G$ and the restriction of $\text{gyr}[a, b]$ to $X$ equals the identity
map on $X$ for all $a, b\in X$.
\end{proposition}

In this paper, $\text{gyr}[a,b](V)$ denotes $\{\text{gyr}[a,b](v): v\in V\}$.

\begin{theorem}\cite{Ung}\label{the1.3}
Let $(G, \oplus)$ be a gyrogroup. Then, for any $a, b, c\in G$ we have
\begin{enumerate}
\item[(1)] $(a\oplus b)\oplus c=a\oplus(b\oplus \text{gyr}[b, a]c);$\quad \quad\quad\quad Right Gyroassociative Law
\item[(2)] $\text{gyr}[a, b]=\text{gyr}[a, b\oplus a];$\quad \quad\quad\quad\quad\quad\quad\quad Right Loop Property
\item[(3)] $(\ominus a)\oplus(a\oplus b)= b$;
\item[(4)] $(a\ominus b)\boxplus b= a$;
\item[(5)] $(a\boxminus b)\oplus b= a$;
\item[(6)] $\text{gyr}[a, b](c)=\ominus(a\oplus b)\oplus (a\oplus (b\oplus c))$;
\item[(7)] $\ominus(a\oplus b)=\text{gyr}[a, b](\ominus b\ominus a)$;\quad \quad\quad \quad\quad\quad Gyrosum Inversion
\item[(8)] $\text{gyr}[a, b](\ominus x)=\ominus \text{gyr}[a, b]x$;
\item[(9)] $\text{gyr}^{-1}[a, b]=\text{gyr}[b, a]$; \quad \quad\quad \quad\quad\quad Inversive symmetry
\item[(10)] $\ominus(a\boxplus b)= (\ominus b)\boxplus(\ominus a)$; \quad\quad\quad The Cogyroautomorphic Inverse Theorem
\item[(11)] $\text{gyr}[\ominus a, \ominus b]=\text{gyr}[a, b]$; \quad\quad\quad Even symmetry
\item[(12)] $\text{gyr}[a, 0]=\text{gyr}[0, b]=I$.
\end{enumerate}
\end{theorem}

\begin{theorem}\cite{Ung}\label{the2.5com}
 Let $(G, \oplus)$ be a gyrocommutative gyrogroup. Then, for any $a, b, c\in G$ we have
 \begin{enumerate}
\item[(1)] $\ominus(a\oplus b)=\ominus a\ominus b;$\quad \quad\quad\quad Gyroautomorphic Inverse Property
\item[(2)] $a\boxplus b=b\boxplus a;$
\item[(3)] $a\boxplus b=a\oplus((\ominus a\oplus b)\oplus a)$.
\end{enumerate}
\end{theorem}

\begin{definition}\cite{ST}\label{NS}
A subgyrogroup $N$ of a gyrogroup $G$ is
normal in $G$, written $N\trianglelefteq G$, if it is the kernel of a gyrogroup homomorphism of $G$.
\end{definition}

\begin{theorem}\cite{ST}\label{the2.5}
Let $N$ be a subgyrogroup of a gyrogroup $G$. Then $N$ is a
normal subgyrogroup in $G$ if and only if
$a\oplus(N\oplus b)=(a\oplus b)\oplus N=(a\oplus N)\oplus b$
for all $a, b\in G$.
\end{theorem}

Since in Topology 'normal' refers to a separation property of spaces, we will use the term
'invariant' to denote this property of subgyrogroups.

\begin{definition}\cite{Atip}
A triple $(G, \tau,  \oplus)$ is called a {\it topological gyrogroup} if and only if
\begin{enumerate}
\item[(1)] $(G, \tau)$ is a topological space;
\item[(2)] $(G, \oplus)$ is a gyrogroup;
\item[(3)] The binary operation $\oplus:G \times G\rightarrow G$ is continuous where $G\times G$ is endowed with the product topology
and the operation of taking the inverse $\ominus(\cdot ) : G  \rightarrow G $, i.e. $x\rightarrow\ominus x$, is continuous.
\end{enumerate}
\end{definition}

If a triple $( G, \tau,  \oplus)$ satisfies the first two conditions and its binary operation is continuous, we call such
triple a {\it paratopological gyrogroup} \cite{Atip1}. Sometimes we will just say that $G$ is a topological gyrogroup (paratopological gyrogroup) if the binary operation and the topology are clear from the context.

\begin{definition}\cite{BL}\label{defst}
Let $( G, \tau,  \oplus)$ be a topological gyrogroup. We say that $G$ is a strongly topological gyrogroup if
there exists a neighborhood base $\mathscr{U}$ of the identity 0 in $G$ such that, for every $U\in \mathscr{U}$,
$\text{gyr}[x, y](U)=U$ for any $x, y\in G$.
\end{definition}

Similarly, we called a paratopological gyrogroup $( G, \tau,  \oplus)$ a {\it strongly paratopological gyrogroup} if
there exists a neighborhood base $\mathscr{U}$ of the identity 0 in $G$ such that, for every $U\in \mathscr{U}$,
$\text{gyr}[x, y](U)=U$ for any $x, y\in G$.

\begin{proposition}\cite{Atip1}\label{pro23}
Let $G$ be a paratopological gyrogroup, $x, y\in G$ and $A, B\subseteq G$.
\begin{enumerate}
\item[(1)] The left translation $L_{x}: G \rightarrow G$, where $L_{x}(y) = x \oplus y$ for every $y\in G$, is
homeomorphism;
\item[(2)] $A$ is closed if and only if $x\oplus A$ is closed;
\item[(3)] $A$ is open if and only if $x\oplus A$ and $B\oplus A$ are open;
%\item[(4)] The right translation $R_{x}: G \rightarrow G$, where $R_{x}(y) = y\oplus x$ for every $y\in G$, is continuous;
%\item[(5)] $A$ is open then $A\oplus x$ and $A\oplus B$ are open.
\end{enumerate}
\end{proposition}

\begin{proposition}\label{pro23s}
Let $G$ be a paratopological gyrogroup, $x, y\in G$ and $A, B\subseteq G$.
\begin{enumerate}
\item[(1)] $\text{gyr}[x,y]: G \rightarrow G$, for every $x, y\in G$, is
homeomorphism;
\item[(2)] $A$ is closed if and only if $\text{gyr}[x,y](A)$ is closed;
\item[(3)] $A$ is open if and only if $\text{gyr}[x,y](A)$ is open.
\end{enumerate}
\end{proposition}
\begin{proof}
By definition, $\text{gyr}[x,y]$ is bijective. Moreover, the gyrator identity provides that
$\text{gyr}[x,y]=L_{\ominus(x\oplus y)}\circ L_{x}\circ L_y$
which is a homeomorphism by Proposition \ref{pro23}. So (2) and (3) are true.
\end{proof}
\begin{proposition}\label{pro22}
Let $G$ be a paratopological gyrogroup and $U$ be a neighborhood of the identity 0. Then
there is an open neighborhood $V$ of 0 such that $\overline{V}\oplus \overline{V}\subseteq \overline{U}$.
\end{proposition}
\begin{proof}
For $G$ is a paratopological gyrogroup, then $op_2:G\times G\rightarrow G$ defined by $op_2(x,y)=x\oplus y$ is continuous.
Because $0\oplus 0=0$, and $U$ is a neighborhood of the identity 0,
there exist a neighborhood $V$ of $0$ such that $\overline{V}\oplus \overline{V}=op_2(\overline{V}\times\overline{V})
=op_2(\overline{V\times V})\subseteq \overline{op_2(V\times V)}= \overline{V\oplus V}\subseteq \overline{U}$.
\end{proof}

\begin{lemma}\label{lem25}
Let $G$ be a paratopological gyrogroup and $\mathscr{U}$ be the neighborhood base
at 0 of $G$.
Then for $B=\bigcap\{\overline{U}:U\in\mathscr{U}\}$,
$\text{gyr}[a, b](B)=B$ for any $a, b\in G$.
\end{lemma}
\begin{proof}
For $a, b\in G$, suppose $f(x)=\text{gyr}[a, b](x)$ for any $x\in G$. By Proposition \ref{pro23s}, $f$ is homeomorphism.
Since $f(0)=\text{gyr}[a, b](0)=0$, for $U\in \mathscr{U}$,
there exists $V\in \mathscr{U}$ such that $\text{gyr}[a, b](\overline{V})=f(\overline{V})\subseteq \overline{f(V)}\subseteq\overline{U}$.
It follows that $\text{gyr}[a, b](B)\subseteq B$, for each $a, b\in G$.

It is obvious that $f^{-1}(x)=\text{gyr}[b, a](x)$ by Theorem \ref{the1.3} (9),
which is continuous. Since $f^{-1}(0)=\text{gyr}[b, a](0)=0$,
for $U_1\in \mathscr{U}$,
there exists $V_1\in \mathscr{U}$ such that $\text{gyr}[b, a](\overline{V_1})=f^{-1}(\overline{V_1})\subseteq \overline{f^{-1}(V_1)}\subseteq\overline{U_1}$.
It follows that $\text{gyr}[b, a](B)\subseteq B$, for each $a, b\in G$.
Thus we have $B\subseteq \text{gyr}[a, b](B)$.
So we get $\text{gyr}[a, b](B)=B$ for any $a, b\in G$.
\end{proof}

%\begin{proposition}\cite{BL}
%Let $(G, \tau,\oplus)$ be a topological gyrogroup. Suppose that $H$ is a compact subgyrogroup of $G$, and
%$S$ is a closed subgyrogroup of $G$. Then $H\oplus S$ and $S\oplus H$ are all closed in $G$.
%\end{proposition}

\begin{theorem}\cite{JX1}\label{the2.8}
Let $G$ be a Hausdorff topological gyrogroup and $\mathcal{U}$ an open base at the neutral element $0$ of $G$.
The following conditions hold:
\begin{enumerate}
\item[(8)] for every $U\in\mathcal{U}$ and $x\in G$, there exists $V\in \mathcal{U}$ such that $V\boxplus x\subseteq x\oplus U$ and $x \oplus V \subseteq x\boxplus U$;
\item[(9)] for every $U\in\mathcal{U}$, there exists $V\in \mathcal{U}$ such that $\ominus V\subseteq U$.
\end{enumerate}
\end{theorem}

\begin{proposition}\label{pro2.11}
Let $(G, \tau,\oplus)$ be a paratopological gyrogroup, $F$ a compact subset
of $G$, and $O$ an open subset of $G$ such that $F\subseteq O$.
Then there exists an open neighborhood $V$ of the
identity element 0 such that $F \oplus V \subseteq O$ and $V\oplus F\subseteq O$.
\end{proposition}
\begin{proof}
Since $\oplus:G\times G\rightarrow G$ is continuous in paratopological gyrogroup $G$,
$\oplus^{-1}(O)$ is an open set in $G\times G$ and $\{0\}\times F\subseteq \oplus^{-1}(O)$ for $F\subseteq O$.
Note that $\{0\}\times F$ is compact in $G\times G$, there exist open sets $V_1, W$ in $G$ such that $\{0\}\times F\subseteq V_1\times W\subseteq \oplus^{-1}(O)$. Thus $V_1\oplus F=\oplus(V_1\times F)\subseteq\oplus(V_1\times W)\subseteq O$.

Similarly, one can find an open set $V_2$ in $G$ such that $F\oplus V_2\subseteq O$.
Take $V=V_1\cap V_2$. Then we verify that $F \oplus V \subseteq O$ and $V\oplus F\subseteq O$.
\end{proof}

%\begin{lemma}\label{lem17s}
%Let the neighborhood base
%$\mathscr{U}$ at 0 of $G$ witness that $G$ is a strongly paratopological gyrogroup.
%Then we have $U\oplus a\subset U\boxplus a$,
%$U\ominus a\subset U\boxminus a$ for each $a\in G, U\in \mathscr{U}$.
%\end{lemma}
%\begin{proof}
%By Definitions \ref{defbox} and \ref{defst}, we can get
%$U\boxplus a= \bigcup_{u\in U}u\boxplus a=\bigcup_{u\in U}u\oplus \text{gyr}[u,\ominus a]a
%=\bigcup_{u\in U}u\oplus \bigcup_{u\in U}\text{gyr}[u,\ominus a]a=
%U\oplus \bigcup_{u\in U}\text{gyr}[u,\ominus a]a\supset U\oplus a$ for $U$ is an open neighborhood  of the neutral element 0
%and $\text{gyr}[0,\ominus a]a=a$. By the same way we can get $U\ominus a\subset U\boxminus a$.
%\end{proof}

\begin{lemma}\label{lem17}
Let the neighborhood base
$\mathscr{U}$ at 0 of $G$ witness that $G$ is a strongly paratopological gyrogroup.
Then we have
$a\boxplus U\subseteq a\oplus U$ and $a\boxminus U\subseteq a\ominus U$ for each $a\in G, U\in \mathscr{U}.$
\end{lemma}
\begin{proof}
By Definitions \ref{defbox} and \ref{defst}, we can get
$a\boxplus U= \bigcup_{u\in U}a\boxplus u=\bigcup_{u\in U}a\oplus \text{gyr}[a,\ominus u]u=a\oplus \bigcup_{u\in U}\text{gyr}[a,\ominus u]u\subseteq a\oplus U$
and $a\boxminus U= \bigcup_{u\in U}a\boxminus u=\bigcup_{u\in U}a\ominus \text{gyr}[a, u]u
=a\ominus \bigcup_{u\in U}\text{gyr}[a, u]u\subseteq a\ominus U$, for each $a\in G, U\in \mathscr{U}$.
\end{proof}

\begin{lemma}\label{lem2.13s}
Let the neighborhood base
$\mathscr{U}$ at 0 of $G$ witness that $G$ is a strongly paratopological gyrogroup.
Then for each $a, b\in G, U_1, U_2\in \mathscr{U}$ we have
$U_1\oplus U_2\in \mathscr{U}$.
\end{lemma}
\begin{proof}
By Definition \ref{defst}, we have
$\text{gyr}[a,b](U_1\oplus U_2)=\text{gyr}[a,b](U_1)\oplus \text{gyr}[a,b](U_2)
=U_1\oplus U_2$, for each $a, b\in G, U_1, U_2\in \mathscr{U}$,
which implies $U_1\oplus U_2\in \mathscr{U}$.
\end{proof}

\begin{lemma}\label{lem18}
Let the neighborhood base
$\mathscr{U}$ at 0 of $G$ witness that $G$ is a strongly paratopological gyrogroup.
Then we have
$(a\oplus U)\oplus W= a\oplus (U\oplus W)$ for each $a\in G, U, W\in \mathscr{U}.$
\end{lemma}
\begin{proof}
By Definition \ref{defst} and Theorem \ref{the1.3} (1), we have
$(a\oplus U)\oplus W=a\oplus (U\oplus \bigcup_{u\in U}\text{gyr}[u,a]W)\subseteq a\oplus (U\oplus W)$  and
$a\oplus (U\oplus W)=(a\oplus U)\oplus \bigcup_{u\in U}\text{gyr}[a,u]W\subseteq (a\oplus U)\oplus W$,
for each $a\in G, U, W\in \mathscr{U}.$ Thus we get the result.
\end{proof}

\begin{lemma}\label{lem19}
Let the neighborhood base
$\mathscr{U}$ at 0 of $G$ witness that $G$ is a strongly paratopological gyrogroup.
If $U\oplus V\subseteq W$, then $\ominus V\ominus U\subseteq \ominus W$, for each $W, U, V\in\mathscr{U}$.
\end{lemma}
\begin{proof}
By Theorem \ref{the1.3} (7), we have
$\ominus(U\oplus V)=\bigcup_{u\in U, v\in V}\text{gyr}[u,v](\ominus v\ominus u)\subseteq\ominus W$.
We can get $\text{gyr}[u,v](\ominus v\ominus u)\in\ominus W$ for every $u\in U, v\in V$.
By Theorem \ref{the1.3}(8)(9), $\ominus v\ominus u\in\text{gyr}[v,u](\ominus W)=\ominus \text{gyr}[v,u](W)=\ominus W.$
Thus we get $\ominus V\ominus U\subseteq \ominus W$.
\end{proof}

\begin{lemma}\label{lem2.12}
Let the neighborhood base
$\mathscr{U}$ at 0 of $G$ witness that $G$ is a strongly paratopological gyrogroup.
Then for each $W\in\mathscr{U}$, there exists $U\in\mathscr{U}$ such that $U\oplus (U\oplus U) \subseteq W$ and $(\ominus U\ominus U)\ominus U\subseteq\ominus W$.
\end{lemma}
\begin{proof}
Since operator $\oplus$ is continuous in $G$, for each $W\in\mathscr{U}$
 we can find neighbourhoods $U_1, V$ of 0 such that
$U_1\oplus V \subseteq W$. And for $V\in\mathscr{U}$
 there exists an open neighbourhood $V_1$ of 0 such that
$V_1\oplus V_1\subseteq V$. So $\ominus V\ominus U_1 \subseteq \ominus W$ and $\ominus V_1\ominus V_1\subseteq \ominus V$
by Lemma \ref{lem19}.
Let $U=U_1\cap V_1$. Thus we can get $U\oplus (U\oplus U) \subseteq W$ and $(\ominus U\ominus U)\ominus U\subseteq\ominus W$.
\end{proof}

\begin{lemma}\label{lem2.13}
Let the neighborhood base
$\mathscr{U}$ at 0 of $G$ witness that $G$ is a strongly paratopological gyrogroup.
If $V\oplus V\subseteq U$ where $U, V\in \mathscr{U}$,
then $\ominus(\overline{\ominus V})\subseteq U$.
\end{lemma}
\begin{proof}
We show that $\overline{\ominus V}\subseteq \ominus U$.
 Let $x\in \overline{\ominus V}$.  Then $(x\oplus V)\cap(\ominus V)\neq\emptyset$.
 Therefore there exist $v_1, v_2\in V$ such that $x\oplus v_1=\ominus v_2$ and $x=
 \ominus v_2 \boxminus v_1\in\ominus V \boxminus V
\subseteq \ominus V \ominus V$
 by Lemma \ref{lem17}.
For $\ominus V \ominus V\subseteq \ominus U$ by Lemma \ref{lem19}, we get $x\in \ominus U$.
\end{proof}

Let $(G, \tau,\oplus)$ be a paratopological gyrogroup and $H$ a $L$-subgyrogroup of $G$.
 It follows from \cite[Theorem 20]{Suk3} that $G/H =\{a\oplus H: a\in G\}$ is a partition of $G$. We denote
by $\pi$ the mapping $a\mapsto a\oplus H$ from $G$ onto $G/H$. Clearly, for each $a\in G$,
we have $\pi^{-1}(\pi(a))= a\oplus H$,
for each $a \in G$. Denote by $\tau(G)$ the topology of $G$. In the left cosets
$G/H$ of the gyrogroup $G$, we define a topology $\widetilde{\tau}=\tau(G/H)$ of subsets as follows:
$$\widetilde{\tau}=\tau(G/H) = \{O\subseteq G/H : \pi^{-1}(O)\in\tau(G)\}.$$

A continuous mapping $f:X\rightarrow Y$ is {\it perfect} if $f$ is a closed mapping and all fibers
$f^{-1}(y)$ are compact subsets of $X$.

\begin{proposition}\label{pro2.12}
Let $(G, \tau,\oplus)$ be a paratopological gyrogroup and $H$ a $L$-subgyrogroup of $G$.
Then the natural homomorphism
$\pi$ from a paratopological gyrogroup $G$ to its quotient topology on $G/H$ is an open and continuous mapping.
\end{proposition}

\begin{proof}
The continuity of the map $\pi$ is obvious. If $U\subseteq G$ is an open set then $\pi^{-1}(\pi(U))=U\oplus H$ and
hence $\pi(U)$ is open.
\end{proof}

\begin{proposition}\label{pro27s}
Let $(G, \tau,\oplus)$ be a paratopological gyrogroup and $H$ a $L$-subgyrogroup of $G$. If $H$ is a compact subgyrogroup
of $G$, then the quotient mapping $\pi$ of $G$ onto the quotient space $G/H$ is perfect.
\end{proposition}
\begin{proof}
Let $F$ be a closed subset of the gyrogroup $G$. Let $\widetilde{x}\in G/H\setminus \pi(F)$.
Consider an arbitrary point $x\in \pi^{-1}(\widetilde{x})$. Then $(x\oplus H)\cap F=\emptyset$.
By Proposition \ref{pro2.11} there exists an open neighborhood $U$ of the unit such that $(U\oplus(x\oplus H))\cap F=\emptyset$.
Then $\widetilde{x}\in\pi(U\oplus x)$ and $\pi(U\oplus x)=\pi(U)\oplus \pi(x)=\pi(U)\oplus(\pi(x)\oplus \pi(H))=\pi(U\oplus(x\oplus H))\cap\pi(F)=\emptyset$ thus the map $\pi$ is closed.  Furthermore, if $y\in G/H$
and $\pi(x)=y$ for some $x\in G$, we obtain that $\pi^{-1}
(y)=x\oplus H$ is a compact subset of $G$. Hence, the fibers of $\pi$
are compact. Thus $\pi$ is perfect.
\end{proof}
\section{compact (strongly) paratopological gyrogroups and locally compact (strongly) paratopological gyrogroups}
%Let $G$ be a paratopological gyrogroup. For $A, B\subset G$ we put $<A, B>=
%\{x\in G:A\oplus x\subset B\}$. Let $I(A, B)$ be the interior of the set $<A, B>$ in the space $G$, and
%$\Phi(A, B)=<A, B>\setminus I(A, B)$. Note that if $B$ is closed, then $<A, B>$ is closed and $\Phi(A, B)$ is
%closed and nowhere dense in $G$.

%For a paratopological gyrogroup $(G, \tau,\oplus)$, there exists an open neighborhood base $\mathscr{U}$ of 0
%satisfies conditions in \cite[Theorem 4.2]{JX1}. In addition, if paratopological gyrogroup $G$ satisfying
%conditions (8) and (9) in \cite[Theorem 4.4]{JX1}, then $G$ is a topological gyrogroup.
%To consider the inverse mapping's continuity in strongly paratopological gyrogroups, we give the following
%proposition.

We give the first of our non-trivial claims on the inverse mapping's continuity in paratopological gyrogroups.

\begin{theorem}\label{the3.2}
A compact Hausdorff paratopological gyrogroup $G$ is a topological gyrogroup.
\end{theorem}
\begin{proof}
Let 0 be the neutral element of $G$. Since $G$ is Hausdorff, the set
$M=\{(x,y)\in G\times G:x\oplus y=0\}$ is closed in $G\times G$.

Let $F$ be any closed subset of $G$, and $P=(G\times F)\cap M$.
Then $F$ and $G\times F$ are compact, $P$ closed in $G\times F$, since $M$ is closed, and, therefore, $P$ is compact.
It is true that $(x, y)\in P$ if and only if $y\in F$ and $x\oplus y=0$, that is, $x=\ominus y$.
It follows that the image of $P$ under the natural projection of $G\times G$ onto the first factor $G$ is precisely $\ominus F$.
Since $P$ is compact and the projection mappings are continuous, we conclude that $\ominus F$ is compact, and
therefore, closed in $G$. Thus, the inverse operation in $G$ is continuous. Hence $G$ is a topological gyrogroup.
\end{proof}

It is natural to extend Theorem \ref{the3.2} to locally compact Hausdorff paratopological gyrogroups,
We pose the following problem.
\begin{question}
Is a locally compact Hausdorff paratopological gyrogroup $G$ with a countable base a topological gyrogroup?
\end{question}

Indeed, Theorem \ref{the3.2} can be extended to locally compact Hausdorff strongly paratopological gyrogroups with a slightly more involved argument.
%Let's now succinctly prove that the inverse operation is continuous in any locally compact Hausdorff strongly paratopological gyrogroup.

\begin{lemma}\label{lem20}
Let the neighborhood base
$\mathscr{U}_1$ at 0 of $G$ witness that $G$ is a strongly paratopological gyrogroup and a family $\mathscr{U}=\{U_n: n\in \omega\}\subseteq \mathscr{U}_1$, and $\{x_n: n\in \omega\}$ is a sequence of
points in $G$ such that $x_n\in U_n$ for each $n\in \omega$, and the next conditions are satisfied:
\begin{enumerate}
\item[(1)] $\overline{U_{n+1}\oplus U_{n+1}}\subseteq U_n$ for each $n\in \omega$;
\item[(2)] the sequence $\{y_k:k\in\mathbb{N}\}$, where $y_k=(((x_1\oplus x_2)\oplus x_3)\oplus\cdot\cdot\cdot\oplus x_{k-1})\oplus x_k$,
has an accumulation point $y$ in $G$.
\end{enumerate}
Then there exists $k\in \omega$ such that $\ominus x_{k+1}\in U_0$.
\end{lemma}
\begin{proof}
Since $y\oplus U_1$ is a neighbourhood of $y$, there exists $k\in \mathbb{N}$ such that $y_k\in y\oplus U_1$. Put
$z=\text{gyr}[\ominus y_{k},y_{k+1}](\ominus y_{k+1}\oplus y)$. For $y_{k+1}=y_{k}\oplus x_{k+1}$, we can get $x_{k+1}=\ominus y_{k}\oplus y_{k+1}$.
Thus
\begin{align*}
&\ominus x_{k+1}=\ominus(\ominus y_{k}\oplus y_{k+1})
\\&=\text{gyr}[\ominus y_{k},y_{k+1}](\ominus y_{k+1}\oplus y_k)\quad \text{by Theorem \ref{the1.3} (7)~}
\\&\in\text{gyr}[\ominus y_{k},y_{k+1}](\ominus y_{k+1}\oplus (y\oplus U_1))
\\&=\text{gyr}[\ominus y_{k},y_{k+1}]((\ominus y_{k+1}\oplus y)\oplus (\text{gyr}[\ominus y_{k+1},y]U_1)) \quad \text{by Definition \ref{Def:gyr}~}
\\&\subseteq\text{gyr}[\ominus y_{k},y_{k+1}]((\ominus y_{k+1}\oplus y)\oplus U_1)\quad \quad\quad\quad\text{by Definition \ref{defst}~}
\\&=\text{gyr}[\ominus y_{k},y_{k+1}](\ominus y_{k+1}\oplus y)\oplus \text{gyr}[\ominus y_{k},y_{k+1}](U_1)
\\&=\text{gyr}[\ominus y_{k},y_{k+1}](\ominus y_{k+1}\oplus y)\oplus U_1\quad \quad\quad\quad\text{by Definition \ref{defst}~}
\\&=z\oplus U_1.
\end{align*}
Since by condition (2) the sequence $\{y_m:m\in\mathbb{N}\}$
has an accumulation point $y$ in $G$, for $k\in \mathbb{N}$
the sequence $\{\text{gyr}[\ominus y_{k},y_{k+1}](\ominus y_{k+1}\oplus y_m):m\in\mathbb{N}\}$
has an accumulation point $\text{gyr}[\ominus y_{k},y_{k+1}](\ominus y_{k+1}\oplus y)=z$ in $G$ by Propositions \ref{pro23} and \ref{pro23s}.

For each $m>k+2$,
$y_m=y_{k+1}\oplus x_{k+2}\oplus\cdot\cdot\cdot\oplus x_m$.

Since the neighborhood base
$\mathscr{U}_1$ at 0 of $G$ witness that $G$ is a strongly paratopological gyrogroup and the family $\mathscr{U}\subseteq \mathscr{U}_1$, we have

\begin{align*}
&y_m\in (((y_{k+1}\oplus U_{k+2})\oplus U_{k+3}) \oplus\cdot\cdot\cdot\oplus U_{m-1})\oplus U_m
\\&=(((y_{k+1}\oplus U_{k+2})\oplus U_{k+3}) \oplus\cdot\cdot\cdot\oplus U_{m-2})\oplus (U_{m-1}\oplus U_m) \quad\quad\text{by Lemma \ref{lem18}}
\\&=(((y_{k+1}\oplus U_{k+2})\oplus U_{k+3}) \oplus\cdot\cdot\cdot\oplus U_{m-3})\oplus (U_{m-2}\oplus(U_{m-1}\oplus U_m)) ~\text{by Lemmas \ref{lem2.13s}, \ref{lem18}}
\\&=(((y_{k+1}\oplus U_{k+2})\oplus U_{k+3}) \oplus\cdot\cdot\cdot\oplus U_{m-3})\oplus ((U_{m-2}\oplus U_{m-1})\oplus U_m) \quad\quad\text{by Lemma \ref{lem18}}
\\& \cdot\cdot\cdot\cdot\cdot\cdot
\\&= y_{k+1}\oplus (U_{k+2}\oplus\cdot\cdot\cdot\oplus U_m).
\end{align*}

It follows from condition (1) of the lemma that,
$$\ominus y_{k+1}\oplus y_{m}\in U_{k+2}\oplus\cdot\cdot\cdot\oplus U_m\subseteq U_{k+1}.$$
Therefore, $\text{gyr}[\ominus y_{k},y_{k+1}](\ominus y_{k+1}\oplus y_m)\in\text{gyr}[\ominus y_{k},y_{k+1}](U_{k+1})
\subseteq U_{k+1}$. So we can get
$z\in \overline{U_{k+1}}\subset U_k$, which implies that
$$\ominus x_{k+1}\in z\oplus U_1\subseteq U_k\oplus U_1\subseteq U_0.$$
This finishes the proof.
\end{proof}

\begin{lemma}\label{pro3.1}
Let the neighborhood base
$\mathscr{U}$ at 0 of $G$ witness that $G$ is a strongly paratopological gyrogroup.
Then $G$ is a strongly topological gyrogroup if and only if the inverse operation is continuous at the identity 0.
%If for every $U\in\mathscr{U}$, there exists $V\in \mathscr{U}$ such that $\ominus V\subset U$,
%then $G$ is a strongly topological gyrogroup.
\end{lemma}
\begin{proof}
%Since $(G, \tau,\oplus)$ is a strongly paratopological gyrogroup with a neighborhood base
%$\mathscr{U}$ at 0, we assume $\mathscr{U}$
%satisfies conditions in \cite[Theorem 4.2]{JX1}.
%To prove $(G, \tau,\oplus)$ be a strongly topological gyrogroup, it just need to verify that $G$ satisfies
%conditions (8) and (9) in \cite[Theorem 4.4]{JX1}.

The 'only if' part is clear. We just need to prove the 'if' part.
For $G$ is a strongly paratopological gyrogroup,
by Lemma \ref{lem17}, we have that $x\boxminus U\subseteq x\ominus U $. Thus we have that $x \oplus U\subseteq x\boxplus U$ by the the following operation:

\begin{align*}
&x\boxminus U\subseteq x\ominus U  \quad\quad\quad\quad\text{by Lemma~\ref{lem17} }
\\&\Rightarrow x\boxplus (\ominus U)=x\boxminus U\subseteq x\ominus U \quad\quad\quad\quad\text{by~} a\boxminus b=a \boxplus(\ominus b)
%\\&\Rightarrow \ominus(x\boxplus (\boxminus U))\subset x\ominus U
\\&\Rightarrow \ominus(U\boxplus(\ominus x))=x\boxplus (\ominus U)\subseteq x\ominus U \quad\quad\quad\quad\text{by Theorem \ref{the1.3} (10)~}
\\&\Rightarrow U\boxplus(\ominus x) \subseteq\ominus(x\ominus U)
\\&\Rightarrow U\boxminus x = U\boxplus(\ominus x)\subseteq\ominus(x\ominus U)\quad\quad\quad\quad\text{by~} a\boxminus b=a \boxplus(\ominus b)
\\&\Rightarrow U\subseteq \ominus(x\ominus U)\oplus x \quad\quad\text{by Theorem~\ref{the1.3} (4)~}
\\&\Rightarrow U \subseteq \bigcup_{u\in U}(\ominus(x\ominus u)\oplus x) \quad\quad\quad\quad\text{by~} \ominus(x\ominus U)\oplus x = \bigcup_{u\in U}(\ominus(x\ominus u)\oplus x)
\\&\Rightarrow U \subseteq \bigcup_{u\in U}\text{gyr}[x,\ominus u]u \quad\quad\text{by Theorem~\ref{the1.3} (6)~}
%\\&\Leftrightarrow&x \oplus V \subset x\boxplus U
\\& \Rightarrow x\oplus U \subseteq \bigcup_{u\in U}x\oplus \text{gyr}[x,\ominus u]u
\\& \Rightarrow x\oplus U \subseteq x\boxplus U \quad\quad\quad\text{by~} x\boxplus u=x\oplus \text{gyr}[x,\ominus u]u \quad\quad\quad (*)
\end{align*}

%&x \oplus V \subset x\boxplus U
%\\&= \bigcup_{u\in U}x\oplus \text{gyr}[x,\ominus u]u
%\\&\Leftrightarrow V \subset \bigcup_{u\in U}\text{gyr}[x,\ominus u]u \quad\quad\text{by Theorem\ref{the1.3} (3)~}
%\\&\Leftrightarrow V \subset \bigcup_{u\in U}(\ominus(x\ominus u)\oplus x)\quad\quad\text{by Theorem\ref{the1.3} (6)~}
%\\&\Leftrightarrow V \subset \ominus(x\ominus U)\oplus x
%\\&\Leftrightarrow V\boxminus x \subset\ominus(x\ominus U) \quad\quad\text{by Theorem\ref{the1.3} (4)~}
%\\&\Leftrightarrow \ominus(V\boxminus x)\subset x\ominus U
%\\&\Leftrightarrow x\boxminus V\subset x\ominus U  \quad\quad\quad\quad\text{by Theorem\ref{the1.3} (10)~}
%\\&\Leftrightarrow \ominus (\ominus x\oplus(x\boxminus V))\subset U.
For each $U\in\mathscr{U}$, $x\in G$, it is obvious that $(x\oplus U)\ominus x$ is a neighborhood of $0$.
Hence there exists $V_1\in \mathscr{U}$ such that $V_1\subseteq (x\oplus U)\ominus x$, which is equivalent to
$$V_1\boxplus x\subseteq x\oplus U.\quad\quad\quad (**)$$

Take any $x\in G$,  $O\in \mathscr{U}$. By (**), there is $U_1\in \mathscr{U}$ such that $U_1\boxplus(\ominus x)\subseteq \ominus x\oplus O$.
For $U_1$, there exists $U_2\in \mathscr{U}$ such that $\ominus U_2\subseteq U_1$. For $U_2$, one can find $V_1\in \mathscr{U}$ such that $x\oplus V_1\subseteq x \boxplus U_2$ by (*). Then we have that
\begin{align*}
\ominus(x\oplus V_1)&\subseteq \ominus(x \boxplus U_2)
\\&=\ominus U_2 \boxplus (\ominus x)
\\&\subseteq U_1\boxplus (\ominus x)
\\&\subseteq \ominus x\oplus O.
\end{align*}
Thus we show that the inverse operation is continuous at any $x\in G$.
This finishes the proof.
\end{proof}

\begin{theorem}\label{the3.6}
If $G$ is a Hausdorff
locally compact strongly paratopological gyrogroup, then $G$ is a strongly topological gyrogroup.
\end{theorem}
\begin{proof}
To prove $G$ is a topological gyrogroup,
it is just to show that the inverse operation is continuous at $0\in G$ by Lemma \ref{pro3.1}.

Let the neighborhood base
$\mathscr{U}$ at 0 of $G$ witness that $G$ is a strongly paratopological gyrogroup. We shall prove that for each $U\in \mathscr{U}$ one can find $V\in \mathscr{U}$ such that $\ominus V\subseteq U$.
Assume the contrary, there is a $U\in \mathscr{U}$ such that for each $V \in \mathscr{U}$, $\ominus V$ is not a subset of $U$. Since $G$ is a Hausdorff
locally compact space, so $G$ is regular. Thus, we can find a $U_0\in \mathscr{U}$ such that $U\supseteq\overline{U_0}$ is compact. $G$ is a paratopological gyrogroup, so we can define a sequence $\{U_n:n\in \omega\}$ of $\mathscr{U}$ such that $\overline{U_{n+1}\oplus U_{n+1}}\subseteq U_n$ for each $n\in \omega$ and there is $x_n\in U_n$ satisfying $\ominus x_n \notin U_0$ for each $n\in \mathbb{N}$. Put $y_k=x_1\oplus x_2\oplus \cdot\cdot\cdot\oplus x_k$, for each $k\in\mathbb{N}$.
Then from Lemmas \ref{lem17} and \ref{lem18} it easily follows all elements $y_k$ are in $U_0$. In fact,

\begin{align*}
&y_k=x_1\oplus x_2\oplus \cdot\cdot\cdot\oplus x_k
\\&\in x_1\oplus U_2\oplus \cdot\cdot\cdot\oplus U_k
\\&= x_1\oplus (U_2\oplus \cdot\cdot\cdot\oplus U_k) \quad\quad\text{by Lemmas \ref{lem2.13s}, \ref{lem18}}
\\&\subseteq x_1\oplus U_1
\\&\subseteq U_1\oplus U_1
\\&\subseteq U_0.
\end{align*}
Since the closure of $U_0$ is compact, there exists an accumulation point $y$ for the sequence $\{y_k:k\in\mathbb{N}\}$ in $G$. Thus by Lemma \ref{lem20} there is a $k\in\omega$ such that $\ominus x_{k+1}\in U_0$. This is a contradiction with $\ominus x_{k+1}\notin U_0$.
Thus we have proved that the inverse operation $\ominus$ is continuous at 0.
This finishes the proof.
\end{proof}

In the latter case we try to remove the Hausdorff restriction in Theorem \ref{the3.6}
To demonstrate this, we need the following theorem, which is inspired by Ravsky's result \cite{Ra}.

\begin{theorem}\label{the3.12}
Let $G$ be a strongly paratopological gyrogroup,
and $H$ be an invariant subgyrogroup of $G$.
If $H$ and $G/H$ are strongly topological gyrogroups, then so is $G$.
\end{theorem}
\begin{proof}
To prove $G$ is a topological gyrogroup,
it is just to show that the inverse operation is continuous at $0\in G$ by Lemma \ref{pro3.1}.

Let the neighborhood base
$\mathscr{U}$ at 0 of $G$ witness that $G$ is a strongly paratopological gyrogroup. We shall prove that for each open neighborhood $U\in\mathscr{U}$ at 0 in $G$ there exists an open neighborhood
$V_2\in\mathscr{U}$
such that $\ominus V_2\subseteq U$.
For each open neighborhood $U\in\mathscr{U}$, there exists an open neighborhood $U_1\in\mathscr{U}$ at 0 in $G$ such that $U_1\oplus U_1\subseteq U$.
For $U_1$, there exists an open neighborhood $V_1\in\mathscr{U}$ such that $V_1\subseteq U_1$, and
$(\ominus V_1\ominus V_1)\cap H\subseteq U_1\cap H\subseteq U_1$, for $H$ is
a topological gyrogroup.
For $V_1$, there exists an open neighborhood $V_2\in\mathscr{U}$ such that $V_2\subseteq V_1$, and
$\ominus V_2\subseteq\ominus(V_2\oplus H)=\pi^{-1}(\ominus\pi(V_2))\subseteq\pi^{-1}(\pi(V_1))=V_1\oplus H$,
for $\pi$ is an open mapping of $G$ onto $G/H$ and $G/H$ is a topological gyrogroup.
If $x\in\ominus V_2$
then there exist elements $v\in V_1, h\in H$ such that $x = v\oplus h$. Then
$h =\ominus v\oplus x\in(\ominus V_1\ominus V_2) \cap H\subseteq U_1$. Therefore $x\in V_1\oplus U_1\subseteq U_1\oplus U_1\subseteq U$.
Thus we have proved that the inverse operation $\ominus$ is continuous.
This finishes the proof.
\end{proof}

\begin{proposition}\label{the26}
Let $G$ be a paratopological gyrocommutative gyrogroup
and $\mathscr{U}$ be the neighborhood base
at 0 of $G$.
If $G$ is locally compact, then $B=\bigcap\{\overline{U}:U\in\mathscr{U}\}$ is a closed invariant
subgyrogroup of $G$.
\end{proposition}
\begin{proof}
It is obvious that $B$ is closed.
Firstly, we shall prove that $B$ is a subgyrogroup.
By Proposition \ref{pro22} we can get $B=\bigcap\{\overline{U}\oplus\overline{U}:U\in\mathscr{U}\}$.
Clearly, for every $U\in \mathscr{U}$, we have that $B\oplus B\subseteq \overline{U}\oplus\overline{U}$.
Hence it follows that $B\oplus B=B$.
Since $G$ is locally compact, $B$ is compact. A nonempty subset $M$ of $B$ is called a right ideal in
$B$ if $M\oplus B\subseteq M$. Since $B$ is compact, closed and $B\oplus B=B$, applying the
Kuratowski-Zorn lemma to the family of all closed right ideals in $B$ ordered by inverse inclusion,
it contains a minimal closed
right ideal, denoted by $H$.
For an arbitrary element $x\in H$, we
have that $x\oplus H\subseteq H\oplus B\subseteq H$.
Follows Lemma \ref{lem25}, it is clear that $(x\oplus H)\oplus B=x\oplus (H\oplus\bigcup_{h\in H}\text{gyr}[h, x]B) =x\oplus(H\oplus B)\subseteq x\oplus H$, i.e., $x\oplus H$ is a right
ideal in $B$. Since $x\oplus H$ is closed in $B$, $x\oplus H\subseteq H$, and $H$ is a minimal right ideal in $B$,
we conclude that $x\oplus H=H$ for each $x\in H$. In particular, $x\oplus(x\oplus H)=H$ for any $x\in H$,
whence it follows that $x\oplus(x\oplus y)=x$ for some $y\in H$ and hence $\ominus x= y\in H$. In its turn,
this implies that $0\in H$, $H= B$, and that $B$ is a subgyrogroup of $G$.

Secondly, we shall prove that $B$ is an invariant subgyrogroup.

{\bf Claim 1.} $a\oplus B= B\oplus a$ for each $a\in G$.

For $a\in G$, suppose $f_1(x)=\ominus a\oplus(x\oplus a)$ for any $x\in G$.
So, $f_1=L_{\ominus a}\circ R_{a}$ which is continuous by Proposition \ref{pro23}.
Since $f_1(0)=\ominus a\oplus(0\oplus a)=0$, for $U\in \mathscr{U}$,
there exists $V\in \mathscr{U}$ such that $f_1(\overline{V})=\ominus a\oplus(\overline{V}\oplus a)\subseteq\overline{U}$.
It follows that $\ominus a\oplus(B\oplus a)\subseteq B$, for each $a\in G$, that is $B\oplus a\subseteq a\oplus B$.

On the other hand, for $a\in G$, suppose $f_2(x)=\ominus a\oplus(a\oplus (a\oplus x)\ominus a)$ for any $x\in G$.
So, $f_2=L_{\ominus a}\circ R_{\ominus a}\circ L_{a}\circ L_{a}$ which is continuous by Proposition \ref{pro23}.
Since $f_2(0)=\ominus a\oplus(a\oplus (a\oplus 0)\ominus a)=0$, for $U_1\in \mathscr{U}$,
there exists $V_1\in \mathscr{U}$ such that $f_2(\overline{V_1})=\ominus a\oplus(a\oplus (a\oplus \overline{V_1})\ominus a)\subseteq\overline{U_1}$.
It follows that $\ominus a\oplus(a\oplus (a\oplus B)\ominus a)\subseteq B$,
for each $a\in G$.
Since $G$ is a gyrocommutative gyrogroup, for each $h\in B$ we have
\begin{align*}
&(a\oplus h)\boxminus a=(a\oplus h)\boxplus(\ominus a)
\\&=\ominus a\boxplus (a\oplus h)\quad\quad\quad\quad\quad\quad\quad\quad\text{ by Theorem \ref{the2.5com}}
\\&=\ominus a\oplus(a\oplus (a\oplus h)\ominus a). \quad\quad\text{by Theorem \ref{the2.5com}}
\end{align*}
So we can get $(a\oplus B)\boxminus a=\ominus a\oplus(a\oplus (a\oplus B)\ominus a)\subseteq B$, which means
$a\oplus B\subseteq B\oplus a$. In conclusion, we get $a\oplus B= B\oplus a$.

{\bf Claim 2.} $(a\oplus B)\oplus b= (a\oplus b)\oplus B$ for each $a, b\in G$.

For $a, b\in G$, suppose $f_3(x)=\ominus (a\oplus b)\oplus((a\oplus x)\oplus b)$ for any $x\in G$.
So, $f_3=L_{\ominus (a\oplus b)}\circ R_{b}\circ L_{a}$ which is continuous by Proposition \ref{pro23}.
Since $f_3(0)=\ominus (a\oplus b)\oplus((a\oplus 0)\oplus b)=0$, for $U_2\in \mathscr{U}$,
there exists $V_2\in \mathscr{U}$ such that $f_3(\overline{V_2})=\ominus (a\oplus b)\oplus((a\oplus \overline{V_2})\oplus b)\subseteq\overline{U_2}$.
It follows that $\ominus (a\oplus b)\oplus((a\oplus B)\oplus b)\subseteq B$, for each $a, b\in G$,
that is $(a\oplus B)\oplus b\subseteq (a\oplus b)\oplus B$.

Also,
for $a, b\in G$, suppose $f_4(x)=\ominus a\oplus(\ominus b\oplus ((b\oplus (a\oplus b\oplus x))\ominus b))$ for any $x\in G$.
So, $f_4=L_{\ominus a}\circ R_{\ominus b}\circ L_{\ominus b}\circ L_{b}\circ L_{a\oplus b}$ which is continuous by Proposition \ref{pro23}.
Since $G$ is a gyrocommutative gyrogroup, for each $h\in B$ we have
\begin{align*}
&((a\oplus b)\oplus h)\boxminus b=((a\oplus b)\oplus h)\boxplus (\ominus b)
\\&=(\ominus b)\boxplus((a\oplus b)\oplus h)\quad\quad\quad\quad\quad\quad\quad\quad\text{ by Theorem \ref{the2.5com}}
\\&=\ominus b\oplus((b\oplus (a\oplus b \oplus h))\ominus b). \quad\quad\text{by Theorem \ref{the2.5com}}
\end{align*}
So $f_4(0)=\ominus a\oplus(\ominus b\oplus ((b\oplus (a\oplus b\oplus 0))\ominus b))=\ominus a\oplus(((a\oplus b)\oplus 0)\boxminus b)=0$, for $U_3\in \mathscr{U}$,
there exists $V_3\in \mathscr{U}$ such that $f_4(\overline{V_3})=\ominus a\oplus(((a\oplus b)\oplus \overline{V_3})\boxminus b)\subseteq\overline{U_3}$.
It follows that $\ominus a\oplus(((a\oplus b)\oplus B)\boxminus b)\subseteq B$, that is $(a\oplus b)\oplus B\subseteq (a\oplus B)\oplus b$.
for each $a, b\in G$.
In conclusion, we get $(a\oplus B)\oplus b= (a\oplus b)\oplus B$ for each $a, b\in G$.

{\bf Claim 3.} $a\oplus (B\oplus b)= (a\oplus b)\oplus B$ for each $a, b\in G$.

For $a, b\in G$, suppose $f_5(x)=\ominus (a\oplus b)\oplus(a\oplus (b\oplus x))=\text{gyr}[a, b](x)$ for any $x\in G$.
By Lemma \ref{lem25}, we can get $\text{gyr}[a, b](B)=B$, that is $\ominus (a\oplus b)\oplus(a\oplus (b\oplus B))=B$.
Thus, $\ominus (a\oplus b)\oplus(a\oplus (B\oplus b))=B$, by Claim 1.
In conclusion, we have $(a\oplus b)\oplus B= a\oplus (B\oplus b)$.

From Claims 1, 2 and 3 it follows that
$B$ is an invariant
subgyrogroup of $G$ by Theorem \ref{the2.5}.
Hence we prove that $B$ is a closed invariant
subgyrogroup of $G$.
\end{proof}

\begin{theorem}\label{the3.17}
If $G$ is locally compact strongly paratopological gyrocommutative gyrogroup, then
$G$ is a strongly topological gyrogroup.
\end{theorem}
\begin{proof}
Let the neighborhood base
$\mathscr{U}$ at 0 of $G$ witness that $G$ is a strongly paratopological gyrogroup.
Since $B=\bigcap\{\overline{U}:U\in\mathscr{U}\}$ is a closed invariant
subgyrogroup of $G$ by Proposition \ref{the26}, the quotient paratopological gyrogroup $G/B$ is a $T_1-$space.
Since $B$ is compact, the quotient homomorphism $\pi:G\rightarrow G/B$ is a
closed mapping by Proposition \ref{pro27s}. So we can get the space $G/B$ is locally compact.
We prove that $G/B$ is Hausdorff. Suppose for a contradiction that two distinct elements $a,b\in G/B$ cannot be
separated by open neighborhoods. Take $x, y\in G$ with $\pi(x)=a$ and $\pi(y) = b$. Then
$(x\oplus B)\cap (y\oplus B)=\emptyset$.
Since $G$ is locally compact, it exists $V\in\mathscr{U}$ such that $\overline{V}$ is compact.
By our assumption, the family
$$\{(x\oplus \overline{U})\cap(y\oplus \overline{U}):U\in\mathscr{U}, U\subseteq V\}$$
of closed subsets of the compact space $x\oplus \overline{V}$ has the finite intersection property, which
in its turn implies that $(x\oplus B)\cap (y\oplus B)\neq\emptyset$.
This contradiction proves that $G/B$ is Hausdorff. Since $G/B$ is a locally compact paratopological gyrogroup, it must be a topological gyrogroup by Theorem \ref{the3.6}.
According to Theorem \ref{the3.12}, $G$ is also a topological gyrogroup.
\end{proof}

\begin{question}
Can the condition 'gyrocommutative' in Theorem \ref{the3.17} be omitted?
\end{question}

\section{pseudocompact strongly paratopological gyrogroups}

Theorem \ref{the3.6} will be extended to pseudocompact (and regular countably compact) paratopological groups in the following section.
The following lemmas can be used to derive additional necessary conditions for a paratopological gyrogroup to be a topological gyrogroup.
\begin{lemma}\label{lem3.7}
Suppose that $G$ is a paratopological gyrogroup, and $U$ is any open neighborhood of
the neutral element 0 in $G$.
Then $\overline{M}\subseteq \ominus U\oplus M$ for each subset $M$ of $G$.
\end{lemma}

\begin{proof}
If $x\notin \ominus U\oplus M$, that is $(U\oplus x)\cap M=\emptyset$,
which implies there exists an open set $U\oplus x$ containing $x$ that has no intersection with $M$.
So, $x\notin \overline{M}$.
\end{proof}

\begin{lemma}\label{lem3.8}
Let the neighborhood base
$\mathscr{U}$ at 0 of $G$ witness that $G$ is a strongly paratopological gyrogroup
and not a topological gyrogroup.
Then there exists an open neighbourhood $U$ of the neutral element 0 in $\mathscr{U}$ such that $U\cap(\ominus U)$
is nowhere dense in $G$, that is, the interior of the closure of $U\cap(\ominus U)$ is empty.
\end{lemma}
\begin{proof}
Since $(G, \tau,\oplus)$ is not a topological gyrogroup,
the inverse operation in $G$ is discontinuous.
Therefore, it is discontinuous at 0 by Lemma \ref{pro3.1}, and we can choose an open neighbourhood $W$ of 0 such that $0\notin \text{int}(\ominus W)$.
Since operator $\oplus$ is continuous in $G$, we can find an open neighbourhood $U$ of 0 such that
$U\oplus (U\oplus U) \subseteq W$. That is $(\ominus U\ominus U)\ominus U\subseteq \ominus W$ by Lemma \ref{lem2.12}.
We claim that the set $U\cap(\ominus U)$ is nowhere dense in $G$.

Assume the contrary. Then there exists a non-empty open set $V$ in $G$ such that $V\subseteq \overline{U\cap(\ominus U)}$.
From Lemma \ref{lem3.7} it follows that $V\subseteq \overline{U\cap(\ominus U)}\subseteq \ominus U\oplus(U\cap(\ominus U))
\subseteq \ominus U\ominus U$. Then $V\ominus U\subseteq(\ominus U\ominus U)\ominus U\subseteq\ominus W$.
Clearly, $V\cap U\neq\emptyset$, and since $V$ is open the set $V\ominus U$ is open in $G$.
Therefore, $0\in V\ominus U\subseteq \text{int}(\ominus W)$, a contradiction.
\end{proof}

The next lemma easily follows from Lemma \ref{lem3.8}.

\begin{lemma}\label{lem3.8s}
Let the neighborhood base
$\mathscr{U}$ at 0 of $G$ witness that $G$ is a strongly paratopological gyrogroup such that $0\in \overline{\text{int}\overline{(\ominus U)}}$, for each $U\in \mathscr{U}$,
 Then $G$ is a strongly topological gyrogroup.
\end{lemma}

\begin{lemma}\label{lem4.4s}
Suppose that $G$ is a $T_1$-paratopological gyrocommutative gyrogroup which is not a topological gyrogroup.
Then, for each compact subset $F$ of $G$ such that $0\notin F$, there exist an open neighborhood
$O(F)$ of $F$ and an open neighborhood $O(0)$ of 0 such that $O(F)\cap(\ominus O(0))=\emptyset$.
\end{lemma}

\begin{proof}
For each $x\in F$, we select an open neighborhood $V_x$ of 0 such that $\ominus x\notin V_x\oplus V_x$.
Then $(\ominus V_x\ominus x)\cap V_x=\emptyset$. We can get $V_x\oplus x=\ominus(\ominus V_x\ominus x)\cap (\ominus V_x)=\emptyset$ by Theorem \ref{the2.5com} (1).
Since $\gamma= \{V_x \oplus x: x\in F\}$ is a family of open sets in $G$ covering the
compact subspace $F$, there exists a finite subset $K$ of $F$ such that $F\subseteq \bigcup\{V_x \oplus x: x\in K\}$. Put
$O(0)=\bigcap\{V_x : x\in K\}$ and $O(F)=\bigcup\{V \oplus x: x\in K\}$. Then $O(0)$ is an open neighborhood
of 0, $O(F)$ is an open neighborhood of $F$, and $O(F)\cap(\ominus O(0))=\emptyset$.
\end{proof}

\begin{theorem}
Suppose that $f$ is a perfect homomorphism of a $T_1$-strongly paratopological gyrocommutative gyrogroup $G$
onto a strongly topological gyrogroup $H$. Then $G$ is also a strongly topological gyrogroup.
\end{theorem}

\begin{proof}
Assume that $G$ is not a strongly topological gyrogroup. and let the neighborhood base
$\mathscr{U}$ at 0 of $G$ witness that $G$ is a strongly paratopological gyrogroup.
Then, according to Lemma \ref{lem3.8s}, there
exists an open neighborhood $U\in \mathscr{U}$ such that 0 is not in $\text{int}(\ominus U)$.
Put $F = f^{-1}f(0)$ and $F_1= F\setminus U$. Since $F_1$ is compact and 0 is not in $F_1$, Lemma \ref{lem4.4s}
implies that there exist an open neighborhood $O(F_1)$ of $F_1$ and an open neighborhood
$O(0)$ of 0 such that $O(F_1)\cap(\ominus O(0))=\emptyset$.

Since $W=O(F_1)\cup U$ is an open neighborhood of $F$ and the mapping $f$ is closed, there
exists an open neighborhood $V$ of $f(0)$ in $H$ such that $f^{-1}(V)\subseteq W$. We can also assume
that $\ominus V=V$, since $H$ is a topological gyrogroup. Then $\ominus f^{-1}(V)=f^{-1}(\ominus V)=f^{-1}(V)\subseteq W$. Finally,
put $W_0=f^{-1}(V)\cap O(0)\cap U$. Clearly, $W_0$ is an open neighborhood of 0 contained in $U$.
We also have $\ominus W_0\subseteq \ominus f^{-1}(V)\subseteq W$
and $\ominus W_0\subseteq\ominus O(0)$.
Since $O(F_1)\cap(\ominus O(0))=\emptyset$,
it follows that $\ominus W_0\subseteq U$.
Therefore, $0\in W_0\subseteq \text{int}(\ominus U)$, a contradiction.
\end{proof}

Here, we demonstrate that every pseudocompact paratopological gyrogroup is a topological gyrogroup.
It is well known that a Tychonoff space $X$ is pseudocompact if and only if every locally finite family
of open sets in $X$ is finite.
To present results in a broad sense, we recall that a topological space
$X$ is called feebly compact if every locally finite family of open sets in $X$ is finite. Therefore,
'feebly compact' is equivalent to 'pseudocompact' for Tychonoff spaces.
This result is improved on A. V. Arhangel'ski\v{\i} and E. A. Reznichenko's results.

\begin{theorem}\label{the3.8}
If a strongly paratopological gyrogroup $G$ is a dense $G_\delta$-set in a regular feebly compact space $X$,
then $G$ is a strongly topological gyrogroup.
\end{theorem}
\begin{proof}
Let the neighborhood base
$\mathscr{U}$ at 0 of $G$ witness that $G$ is a strongly paratopological gyrogroup. Assume the contrary. Then, by Lemma \ref{lem3.8}, there exists an open
neighbourhood $U$ of the neutral element 0 in $\mathscr{U}$ such that $U\cap(\ominus U)$ is nowhere dense.
Let $W\in\mathscr{U}$ such that $\overline{W\oplus W}\subseteq U$.
Put $O=W\setminus \overline{U\cap(\ominus U)}$. Then, clearly,
$O\subseteq W\subseteq \overline{O}$ and $\ominus O\cap U=\emptyset$.
First, we fix a sequence $\{M_n: n\in\omega\}$ of open sets in $X$ such that $G=\bigcap_{n=0}^\infty M_n$.
We are going to define a sequence $\{U_n: n\in\omega\}$ of open subsets of $X$ and a sequence
$\{x_n: n\in\omega\}$ of elements of $G$ such that $x_n\in U_n$, for each $n\in \omega$. Put $U_0=O$, and pick a
point $x_0 \in U_0\cap G$.

Assume now that, for some $n\in\omega$, an open subset $U_n$ of $X$ and a point $x_n\in U_n\cap G$
are already defined. Since $0\in W\subseteq\overline{O}$, we have $x_n\in x_n\oplus\overline{O}=\overline{x_n\oplus O}$.
Since $U_n$ is an open neighbourhood of $x_n$, it follows that $U_n \cap x_n\oplus O\neq\emptyset$. We take $x_{n+1}$  to be any point of $U_n\cap x_n\oplus O$. Note that $x_{n+1}\in G$, since $x_n\oplus O\subseteq G$.
Using the regularity of $X$, we can find an open neighbourhood $U_{n+1}$ of $x_{n+1}$ in $X$ such
that the closure of $U_{n+1}$ is contained in $U_n\cap M_n$, and $U_{n+1}\cap G\subseteq x_n\oplus O$.
The definition of the sets $U_n$ and points $x_n$, for each $n\in\omega$, is complete. Note that $\overline{U_i}\subseteq U_j$ whenever $j<i$. We also have $x_{n+1}\in x_n\oplus O$, for each $n\in\omega$.
Put $F=\bigcap_{n\in\omega}\overline{U_n}$. Clearly, $F\subseteq G$, and $F\neq\emptyset$ since $X$ is feebly compact. The set
$F\oplus W$ is an open neighbourhood of $F$ in $G$. Consider the closure $P$ of $F\oplus W$ in $X$, and let $H$
be the closure of $X\setminus P$ in $X$. Then $H$ is a regular closed subset of $X$, so that $H$ is feebly
compact.

We claim that $H\cap F=\emptyset$. Indeed, assume the contrary, and fix $x\in H\cap F$. Since
$F\oplus W$ is an open neighbourhood of $F$ in $G$, from $x\in F$ it follows that there exists an open
neighbourhood $V$ of $x$ in $X$ such that $V\cap G\subseteq F\oplus W$. Then the density of $G$ in $X$ implies
that $V\subseteq P$, while $x\in V\cap H$ implies that $V\setminus P\neq\emptyset$, which is a contradiction. Thus,
$H\cap F=\emptyset$.

Since $H$ is feebly compact, our definition of $F$ implies that $U_k\cap H=\emptyset$, for some
$k\in\omega$ (we use that $\overline{U_i}\subseteq U_j$ whenever $j<i$). Then $U_k\subseteq P$. Since $x_k\in U_k\cap G$, it follows
that $x_k\in \overline{F\oplus W}$. However, $F\subseteq U_{k+2}\cap G\subseteq x_{k+1}\oplus O\subseteq x_{k+1}\oplus W$. Hence,
$x_k\in \overline{F\oplus W}\subseteq \overline{x_{k+1}\oplus W\oplus W}=x_{k+1}\oplus \overline{(W\oplus W)}\subseteq x_{k+1}\oplus U$, by the definition of $\mathscr{U}$.
Taking into account that $x_{k+1}\in x_k\oplus O$, we obtain that $x_k\in (x_k\oplus O)\oplus U=x_k\oplus (O\oplus U)$. Hence, $0\in O\oplus U$ and
$\ominus O\cap U\neq\emptyset$, which is again a contradiction.
We prove that $G$ is a strongly topological gyrogroup.
\end{proof}

Naturally, Tychonoff spaces are mentioned in the following two corollaries of Theorem \ref{the3.8}.

\begin{corollary}\label{cor3.9}
 Every pseudocompact strongly paratopological gyrogroup is a topological gyrogroup.
 \end{corollary}
\begin{corollary}
 Every \v{C}ech-complete strongly paratopological gyrogroup is a topological gyrogroup.
 \end{corollary}
 Because countably compact spaces are feebly compact, Theorem \ref{the3.8} implies the following fact.
\begin{corollary}\label{cor4.9}
 Every regular countably compact strongly paratopological gyrogroup is a
topological gyrogroup.
 \end{corollary}

\begin{definition}
A paratopological gyrogroup $G$ is called topologically periodic if for each
$x\in G$ and every neighborhood $U$ of the identity there exists an integer $n$ such that
$n\cdot x\in U$.
\end{definition}

\begin{theorem}\label{the4.2}
If a strongly paratopological gyrogroup $G$ is Hausdorff countably compact and topologically periodic,
then $G$ is a strongly topological gyrogroup.
\end{theorem}
\begin{proof}
Let the neighborhood base
$\mathscr{U}$ at 0 of $G$ witness that $G$ is a strongly paratopological gyrogroup,
$U\in\mathscr{U}$ and $\{V_i
|i\in\omega\}$ a family
of neighborhoods of the identity 0 such that $V_i\in\mathscr{U}$, $V_0=U$ and $V_{i+1}\oplus V_{i+1}\subseteq V_i$
for each $i\in\omega$. By lemma \ref{lem2.13} we have $\ominus(\overline{\ominus V_{i+1}})\subseteq V_i$

We show that $F=\cap\{\overline{\ominus V_i}|i\in\omega\}\subseteq U$.
Let $x\in F$, $x\in \overline{\ominus V_i}$ for each $i\in\omega$.
Then $\ominus x\in\ominus(\overline{\ominus V_{i+1}})$ for each $i\in\omega$.
 Choose $n\in \omega$ such that $n\cdot x\in V_1$, then
$(n-1)\cdot(\ominus x)\in \ominus(\overline{\ominus V_{i+1}})\oplus \dots \ominus(\overline{\ominus V_{i+1}})
\subseteq V_i\oplus\dots \oplus V_i$ for each $i\in\omega$.
 Now choose $i_0\in \omega$ such that $V_{i_0}\oplus\dots \oplus V_{i_0}\subset V_1$.
 Then $x=n\cdot x\oplus((n-1)\cdot(\ominus x))\in V_1\oplus V_1\subseteq U$.
 Therefore $F\subset U$. Since $G$ is a
countably compact group, there exist $i_1, \dots, i_k$ such that
$\cap\{\overline{\ominus V_{i_j}}|j=1, \dots, k\}\subseteq U$.
Thus we have proved that the inverse operation $\ominus$ is continuous at 0.
Hence $G$ is a strongly topological gyrogroup by Lemma \ref{pro3.1}.
\end{proof}

\vskip0.9cm

\end{document}